[1994/12/01]% LaTeX date must December 1994 or later
\documentclass[11pt,reqno]{article}

\title{A criterion for transience of multidimensional
branching random walk in random environment}
\author{Sebastian M\"uller}
\usepackage{amsmath,amsthm,amsfonts,bm}
\usepackage{t1enc}
\chardef\bslash=`\\ % p. 424, TeXbook
\newtheorem{thm}{Theorem}[section]
\newtheorem{cor}[thm]{Corollary}
\newtheorem{lem}[thm]{Lemma}

\theoremstyle{definition}
\newtheorem{defn}{Definition}[section]

\theoremstyle{remark}
\newtheorem{rem}{Remark}[section]
\def\orig{o}
\newcommand{\eps}{\varepsilon}

\newcommand{\R}{\mathbb{R}}
\newcommand{\Z}{\mathbb{Z}}
\newcommand{\N}{\mathbb{N}}
\renewcommand{\P}{\mathbb{P}}
\newcommand{\E}{\mathbb{E}}
\newcommand{\U}{\mathfrak{U}}
\renewcommand{\S}{\mathfrak{S}}
\newcommand{\V}{\mathcal{V}}
\newcommand{\M}{\mathcal{M}}
\newcommand{\K}{\mathcal{K}}
\newcommand{\bo}{\bm\omega}

\newcommand{\bs}{\bm\sigma}

\newcommand{\kommentar}[1]{}

\newcommand{\eval}[2][\right]{\relax
  \ifx#1\right\relax \left.\fi#2#1\rvert}

\begin{document}
\maketitle {\abstract We develop a criterion for transience for a
 general model of branching Markov chains. In the case of
multi-dimensional branching random walk in random environment
(BRWRE) this criterion becomes explicit. In particular, we show
that \emph{Condition L} of  \textsc{Comets and Popov}
\cite{comets:05} is necessary and sufficient for transience as
conjectured. Furthermore, the criterion applies to two important
classes of branching random walks and implies that the critical
branching random walk is transient resp. dies out locally.}
\renewcommand{\sectionmark}[1]{}
\section{Introduction}

A branching Markov chain (BMC)  is  a system of particles in
discrete time. The BMC starts with one particle in an arbitrary
starting position $x.$ At each time each particle is independently
replaced by some new particles at some locations according to
given stochastic substitutions rules that may depend on the
location of the substituted particle. Observe that this model is
more general than the model studied in \cite{comets98},
\cite{gantert:04}, \cite{machado:00}, \cite{machado:03},
\cite{mueller06}, where first the particles branch and then,
independently of the branching, move according to an underlying
Markov chain. In some sense the behavior of BMC is more delicate
than the one of Markov chains: while an irreducible Markov chain
is either recurrent or transient, either all or none states are
visited infinitely often, this dichotomy breaks down for BMC.  Let
$\alpha(x)$ be the probability that, starting the BMC in $x$, the
state $x$ is hit infinitely often by some particles. There are
three possible regimes: \emph{transient} $(\alpha(x)=0~ \forall
x)$, \emph{weakly recurrent} $(0<\alpha(x)<1 \mbox{ for some } x)$
and \emph{strongly recurrent} $(\alpha(x)=1~ \forall x),$ compare
with \textsc{Gantert and Müller} \cite{gantert:04} and
\textsc{Benjamini and Peres} \cite{benjamini:94}.

This paper is divided in two parts. First we connect transience
with the existence of superharmonic functions, see Theorem
\ref{thm:1}, and give a criterion for transience of BMC in Theorem
\ref{thm:2}. These two criteria are interesting on their own:
while the first is stated in terms of superharmonic functions, the
second is appropriate to give explicit conditions for transience.
In addition, we see that transience does not depend on the whole
distributions of the substitution rules but only on their first
moments. There are two important classes of BMC where one can
speak of criticality, compare with \cite{gantert:04} and
\cite{pemantle:01}, \cite{stacey:03}. In these cases Theorem
\ref{thm:1} implies that the critical process is transient resp.
dies out locally, compare with Subsection \ref{sub:critical}.

In the second part we follow the line of research of
\textsc{Comets, Menshikov and Popov} \cite{comets98},
\textsc{Comets and Popov} \cite{comets:05}, \textsc{Machado and
Popov} \cite{machado:00}, \cite{machado:03} and the author
\cite{mueller06} and study transience and recurrence of branching
random walk in random environment (BRWRE). In this case we can use
the criterion of the first part, Theorem \ref{thm:2}, to obtain a
classification of BRWRE in transient and strong recurrent regimes.
In particular, we show that the sufficient \emph{Condition L} for
transience  of \textsc{Comets and Popov} \cite{comets:05} is
necessary, too. Classification results of this type were only
known for nearest neighbor BRWRE on $\Z,$ \cite{comets98}, and on
homogeneous trees, \cite{machado:03}. In addition, we show that
transience does not depend on the precise form of the
distributions, but only on the convex hull of their support. Such
phenomena are known for models of this type, compare with
\cite{comets98}, \cite{comets:05}, \cite{machado:00},
\cite{machado:03}, \cite{mueller06}. Our method is quite different
from \cite{comets:05} since we don't analyze the process on the
level of the particles but use superharmonic functions to describe
the process on a more abstract level. The only points where we
really deal with \emph{particles} are the proofs of Theorems
\ref{thm:1} and \ref{thm:0-1}. Furthermore, the only point where
we need the structure of the lattice $\Z^d$ is where the criterion
becomes explicit, see Lemma \ref{lem:4}. All other arguments
immediately apply for BRWRE on Cayley graphs, compare with
\cite{mueller06}. In the special case where branching and movement
are independent the classification in transience and recurrence is
already given in \cite{mueller06}.

The obtained classification result is quite interesting facing the
difficulty of the corresponding questions for  random walks in
random environment of a single particle, compare with
\textsc{Sznitman} \cite{sznitman:02} and \cite{sznitman:03}.

\subsubsection*{Acknowledgment}
The author thanks Nina Gantert and Serguei Popov for valuable
discussions and helpful comments on a previous version of this
paper.

\section{Transience and recurrence for general BMC}

Let $X$ denote the discrete state space of our process. For every
$x\in X$ let
$$\V(x):=\left\{v(x)=(v_y(x))_ {y\in X}: v_y(x)\in\N,~
\sum_{y\in X} v_y(x)\geq 1\right\}$$ be the set of all possible
substitution rules. Furthermore, let $\omega_x$ be a probability
measure on $\V(x)$ and call $\bm\omega :=(\omega_x)_{x\in X}$ the
environment of our process.

The process is defined  inductively: at time $n=0$ we start the
process at some fixed starting position, say $\orig,$ with one
particle. At each integer time the particles are independently
substituted as follows: for each particle in $x\in X$
independently of the other particles and the previous history of
the process a random element of $v(x)\in \V(x)$ is chosen
according to  $\omega_x.$ Then, the particle is replaced by
$v_y(x)$ offspring particles at $y$ for all $y\in Y.$ Thus the BMC
is entirely described through the environment $\bm\omega$ on $X$
and is denoted $(X,\bm\omega).$ In the definition of $\V(x)$ we
demand that the process survives forever: $\sum_{y\in X}
v_y(x)\geq 1$ ensures that each particle has at least one
offspring. This assumption is made for the sake of a better
presentation and to avoid the conditioning on the survival of the
process. A key quantity are the first moments of the substitution
rules $M:=(m(x,y))_{x,y\in X},$ where
$$m(x,y):=\sum_{v\in \V(x)} \omega_x(v) v_y(x)$$ denotes the expected number
of offspring sent to $y$ by a particle in $x.$  Let
$M^n=(m^{(n)}(x,y))_{x,y\in X}$ be the $n$-fold convolution of $M$
with itself and set $M^0=I,$ the identity matrix over $X.$ We will
always assume that $M$ is irreducible:

\noindent {\bf General Assumption:} Let $\bm\omega$ such that $M$
is irreducible, i.e., for all $x,y\in X$ there exists some $k$
such that $m^{(k)}(x,y)>0.$

\begin{rem}
Let $P=\left(p(x,y)\right)_{x,y\in X}$ be the transition kernel of
an irreducible Markov chain on $X.$ Then the assumption on
irreducibility is fulfilled, if for all $x\in X$ we have
$\omega_x(v_y(x)\geq 1)>0$ for all $y$ with $p(x,y)>0.$
\end{rem}

\begin{rem}
The BMC $(X,\bm\omega)$ is a general branching process in the
sense of \textsc{Harris} \cite{harris:63} with first moment $M:$
each particle of the general branching process is characterized by
a parameter $x$ which describes its position in the state space
$X.$
\end{rem}

In order to analyze the process we introduce the following
notations. Let $\eta(n)$ be the total number of particles at time
$n$ and let $x_i(n)$ denote the position of the $i$th particle at
time $n.$ Denote $\P_x(\cdot)=\P(\cdot\mid x_0(1)=x)$ the
probability measure for a BMC started with one particle in $x.$

We define recurrence and transience for  BMC  in analogy to
\cite{benjamini:94} and \cite{gantert:04}:
\begin{defn} Let
\begin{equation}\label{eq:1} \alpha(x):=\P_x\left(
\sum_{n=1}^\infty
\sum_{i=1}^{\eta(n)}\mathbf{1}\{x_i(n)=x\}=\infty\right)
\end{equation} be the probability that $x$ is visited infinitely
often.
 A BMC is {\it transient,} if $\alpha(x)=0$ for all  $x\in X,$
 and
{\it recurrent} otherwise. The recurrent regime is divided into
{\it weakly recurrent,} $\alpha(x)<1$ for some $x$ and
 {\it strongly recurrent,} $\alpha(x)=1$ for all $x\in X.$
\end{defn}
The definition of transience and recurrence does not depend on the
starting position of the process. In fact, due to the
irreducibility,
 $\alpha(x)>0$ and $\alpha(x)=0$ hold
either for all or none $x\in X.$ This can be shown analogously to
\cite{benjamini:94}. In contrast to the model with independent
branching and movement we don't have that $\alpha(x)=1$ either for
all or none $x,$ compare with Example $1$ in \cite{comets:05}.
Observe that our definition differs in the general setting from
the one in \cite{comets:05} but coincide in the case of branching
random walk in random environment, compare with Theorem
\ref{thm:0-1}.

The following criterion for transience in terms of superharmonic
functions is a straightforward generalization of Theorem 3.1 in
\cite{gantert:04}. We give the proof since it makes clear where
the superharmonic functions come into play. In addition, it is the
essential  point where we work on the \emph{particle level.}

\begin{thm}\label{thm:1}
The BMC $(X,\bm\omega)$ is transient if and only if there exists a
positive function $f,$ such that
\begin{equation}\label{eq:1.1}
Mf(x):=\sum_y m(x,y) f(y) \leq f(x)\quad \forall x\in X.
\end{equation}

\end{thm}

\begin{proof}
In analogy to \cite{gantert:04} and \cite{menshikov:97}, we
introduce the following modified version of the BMC. We fix some
site $\orig\in X,$ which stands for the origin of $X.$  The new
process is like the original BMC at time $n=1,$ but is different
for $n>1.$ After the first time step we conceive the origin as
{\it freezing}: if a particle reaches the origin, it stays there
forever without being ever substituted. We denote this new process
{\bf BMC*}. The first moment $M^*$ of BMC* equals $M$ except that
$m^*(\orig,\orig)=1$ and $m^*(\orig,x)=0~\forall x\neq \orig.$ Let
$\eta(n,\orig)$ be the number of frozen particles at position
$\orig$ at time $n.$ We define the random variable $\nu$ as
$$\nu:=\lim_{n\rightarrow\infty} \eta(n,\orig)
\in\N\cup\{\infty\},$$ and write $\E_x \nu$ for the expectation of
$\nu$ given that we start the process with one particle in $x.$

We first show that transience, i.e. $\alpha\equiv 0,$ implies
$\E_{\orig}\nu\leq 1$ and hence, due to the irreducibility of $M,$
that $\E_{x} \nu <\infty$ for all $x.$  We start the BMC in the
origin  $\orig.$ The key idea of the proof is to observe that the
total number of particles ever returning to $\orig$ can be
interpreted as the total number of progeny in a branching process
$(Z_n)_{n\geq 0}.$  Note that each particle has a unique ancestry
line which leads back to the starting particle at time $0$ at
$\orig.$ Let $Z_0:=1$ and $Z_1$ be the number of particles that
are the first particle in their ancestry line to return to
$\orig.$ Inductively we define $Z_n$ as the number of particles
that are the $n$th particle in their ancestry line to return to
$\orig.$ This defines a Galton-Watson process $(Z_n)_{n\geq 0}$
with offspring distribution $Z{\buildrel d \over =}Z_1.$  We have
$Z_1{\buildrel d \over =}\nu$ given that the process starts in the
origin $\orig.$ Furthermore,
$$\sum_{n=1}^\infty Z_n =\sum_{n=1}^\infty \sum_{i=1}^{\eta(n)} \mathbf{1}\{x_i(n)=\orig\}.$$
If $\alpha(\orig)=0,$ then $\sum_{n=1}^\infty Z_n<\infty$ a.s.,
hence $(Z_n)_{n\geq 0}$ is critical or subcritical and $\E_{\orig}
\nu=E[Z]\leq 1.$

In order to show the existence of a superharmonic function it
suffices now to check  that $f(x):=\E_x\nu>0$ satisfies inequality
(\ref{eq:1.1}). For $x$ such that $m(x,\orig)=0$ it is
straightforward to show that even equality holds in
(\ref{eq:1.1}). If $m(x,\orig)>0,$ we have
\begin{eqnarray*}
 f(x)=\E_x\nu &=&  \sum_{y\neq \orig}m(x,y) \E_y \nu
+ m(x,\orig) \cdot 1  \\
   &\geq&   \sum_y m(x,y) \E_y \nu \\
   &=& Mf(x),
\end{eqnarray*}
 since $\E_{\orig} \nu\leq 1.$

Conversely, we consider the  BMC* with origin $\orig$ and define
$$Q(n):=\sum_{i=1}^{\eta(n)} f(x_i(n)),$$
where $x_i(n)$ is the position of the $i$th particle at time $n.$
It turns out  that $Q(n)$ is a positive  supermartingale, so that
it converges a.s. to a random variable $Q_\infty.$ We refer to
\cite{menshikov:97} for the technical details.  Furthermore, we
have
\begin{equation}\label{eq:q}
\nu(\orig)\leq \frac{Q_\infty}{f(\orig)}
\end{equation} for a BMC* with origin $\orig$ started in $\orig.$
 We obtain using Fatou's Lemma
\begin{equation}\label{eq:1.11}
\E_{\orig}\nu\leq \frac{\E_{\orig}Q_\infty}{f(\orig)}\leq
\frac{\E_{\orig}Q(0)}{f(\orig)}=\frac{f(\orig)}{f(\orig)}=
1.\end{equation} Hence, the embedded Galton-Watson process
$(Z_n)_{n\geq 0}$ is critical or subcritical and dies out since
$\P_\orig(\nu(0)<1)>0.$

\end{proof}

Condition (\ref{eq:1.1}) in Theorem \ref{thm:1} suggests that the
spectral radius of the operator $M$ plays a crucial role in
finding a more explicit condition for transience. To pursue this
path let us briefly recall some known properties of irreducible
kernels and of their spectral radii. An operator
$K=\left(k(x,y)\right)_{x,y\in X}$ on $X$ is an irreducible kernel
if $k(x,y)\geq 0$ for all $x,y\in X$ and for all $x,y$ there is
some $l$ such that $k^{(l)}(x,y)>0,$ where
$K^l=\left(k^{(l)}(x,y)\right)_{x,y\in X}$ is  the $l$th
convolution of $K$ with itself.

\begin{defn}\label{def:green_fct}
The Green function of $K$ is the power series
$$G(x,y|z):=\sum_{n=0}^\infty k^{(n)} (x,y) z^n,~ x,y\in X,~
z\in\mathbb{C}.$$
\end{defn}

\begin{lem}\label{lem:2a}
For all $x,y\in X$ the power series $G(x,y|z)$ has the same finite
radius of convergence $R(K)$ given by
$$ R(K):=\left(\limsup_{n\rightarrow\infty}
\left(k^{(n)}(x,y)\right)^{1/n}\right)^{-1}<\infty$$
\end{lem}

\begin{proof}
The fact that the power series defining the functions $G(x,y|z)$
all have the same radius of convergence follows from a system of
Harnack-type inequalities.  Due to the irreducibility of $K$ for
all  $x_1,x_2,y_1,y_2\in X$ there exist some $l_1,l_2\in\N$ such
that we have $k^{(l_1)}(x_1,x_2)>0$ and
 $k^{(l_2)}(y_2,y_1)>0.$  Thus
for every $n\in\N,$
$$k^{(n+l_1+l_2)}(x_1,y_1)\geq k^{(l_1)}(x_1,x_2) k^{(n)}(x_2,y_2)
k^{(k_2)}(y_2,y_1).$$ Consequently, for every $z\in\R^+$
\begin{equation}\label{eq:harnack}
G(x_1,y_1|z)\geq k^{(l_1)}(x_1,x_2) k^{(l_2)}(y_2,y_1)z^{l_1+l_2}
G(x_2,y_2|z).
\end{equation}
It follows that the radius of convergence of $G(x_1,y_1|z)$ is at
least that of $G(x_2,y_2|z).$ The fact that $R(K)<\infty$
 follows from the irreducibility of $K$: let
$l\in\N$ such that $k^{(l)}(x,x)=\eps>0$ then $k^{(nl)}(x,x)\geq
\eps^n$ for every $n\geq 0.$
\end{proof}

\begin{defn}\label{def:spectr_rad}
 The spectral  radius of an irreducible kernel $K$ is defined as
\begin{equation}
\rho(K):=\limsup_{n\rightarrow\infty}
\left(k^{(n)}(x,y)\right)^{1/n}>0.
\end{equation}
\end{defn}

 The following characterization of the spectral radius
in terms of $t$-superharmonic functions is crucial for our
classification:
\begin{lem}\label{lem:1}
$$\rho(K)=\min\{t>0:~\exists\, f(\cdot)>0\mbox{ such that }Kf\leq t f\}$$
\end{lem}
\begin{proof}
We sketch the proof and refer for more details to \cite{pruitt} or
\cite{woess}. We write $S^+(K,t)$ for the collection of all
positive functions $f$ satisfying $Kf\leq t f.$  Observe that a
function in $S^+(K,t)$ is either strictly positive or constant
equal to $0.$ In order to construct a base of the cone $S^+(K,t)$
we fix a reference point $\orig\in X$ and define
$$B(K,t):=\{f\in S^+(K,t): f(\orig)=1\}.$$
If there is some $f\neq 0$ in $S^+(K,t)$ then $\rho(K)\leq t$
since
$$ k^{(n)}(x,x) f(x) \leq K^n f(x)\leq t^n f(x).$$ On the other
hand, observe that for $t> \rho(K)$ we have
$$f(x):=\frac{G(x,\orig|1/t)}{G(\orig,\orig|1/t)}\in B(K,t).$$
The fact that $B(K,\rho(K))=\bigcap_{t>\rho(K)} B^+(K,t)\neq
\emptyset$ follows now from the observation that for all
$t>\rho(K)$ the base $B(K,t)$ is compact in the topology of
pointwise convergence.
\end{proof}

We now obtain as an immediate consequence of Theorem \ref{thm:1} a
classification in terms of the spectral radius of $M.$
\begin{thm}\label{thm:2}
The BMC $(X,\bm\omega)$ is transient if and only if $\rho(M)\leq
1.$
\end{thm}

We can directly show that $\rho(M)< 1$ implies transience. Just
observe that   $G(x,y|1)<\infty$   equals the expected number of
particles visiting $y$ in a BMC started in $x$  and transience
follows immediately. Conversely, the fact that $\rho(M)>1$ implies
the recurrence of the BMC can be also seen by dint of the
interpretation as a general branching process and the fact that
the spectral radius of an infinite irreducible kernel $\rho(M)$
can be approximated by the spectral radii of finite irreducible
kernels. A subset $Y\subset X$ is called irreducible (with respect
to $M$) if the
 operator
\begin{equation}
 M_Y=(m_Y(x,y))_{x,y\in Y}
\end{equation}
defined by
$m_Y(x,y):=m(x,y)$ for all $x,y\in Y$ is irreducible. Notice that
in this case $\rho(M_Y)$ is just the Perron-Frobenius eigenvalue
of $M_Y.$  We find
\begin{equation}\label{eq:specapprox}
\rho(M)=\sup_Y \rho(M_Y),
\end{equation}
where the supremum is over finite and irreducible subsets
$Y\subset X.$ Therefore, if $\rho(M)>1$ then there exists a finite
and irreducible $Y$ such that $\rho(M_Y)>1.$ Now, let us consider
only particles in $Y$ and neglect all the others. We obtain a
supercritical multi-type Galton-Watson process with first moments
$M_Y$ that survives with positive probability since $\rho(M_Y)>1,$
compare with \cite{harris:63}. The subset $Y$ takes the position
of \emph{recurrent seeds} in  \cite{comets:05}. In contrast to
\cite{comets:05} we don't need to construct the seeds explicitly
but use Equation (\ref{eq:specapprox}) in order to make the
criterion explicit for branching random walk in random
environment.

\subsection{The critical BMC is transient}\label{sub:critical}
We have already mentioned in the introduction that the model we
study in this paper is more general than the model where branching
and movement are independent. Let us consider a BMC $(X,P,\mu)$
with independent branching and movement. Here $(X,P)$ is an
irreducible and infinite Markov chain in discrete time and
$$\mu(x)=\left(\mu_k(x)\right)_{k\geq 1}$$ is a sequence of non-negative numbers
satisfying
$$\sum_{k=1}^\infty \mu_k(x)=1 \mbox{ and } m(x):=\sum_{k=1}^\infty k \mu_k(x)<\infty.$$
We define the BMC $(X,P,\mu)$ with underlying Markov chain $(X,P)$
and branching distribution $\mu=(\mu(x))_{x\in X}$ following
\cite{menshikov:97}. At time $0$ we start with one particle in an
arbitrary starting position $x\in X.$   At each time each particle
in position  $x$ splits up according to $\mu(x)$ and the offspring
particles move according to $(X,P).$ At any time, all particles
move and branch independently of the other particles and the
previous history of the process.  Now, let $\bm\omega$ be a
combination of multi-nomial distributions,
$$\omega_x(v):= \sum_k \mu_k(x) Mult\left(k; p(x,y), y\in\{y:p(x,y)>0\}
\right),$$ and hence $(X,\bm\omega)$ has the same distribution as
$(X,P,\mu).$ We immediately obtain the result of
 Theorem 3.2 in \cite{gantert:04}.

\begin{thm}
A BMC $(X,P,\mu)$ with constant mean offspring, i.e.
$m(x)=m~\forall x,$ is transient if and only if $m\leq 1/\rho(P).$
\end{thm}

The general (discrete) BMC can be used to study certain
continuous-time branching random walks. Let us consider the
branching random walk studied for example in \cite{schinazi:93},
\cite{pemantle:01} and \cite{stacey:03}. Let $G=(V,E)$ be a
locally bounded and connected graph with vertex set $V$ and edges
$E.$ The branching random walk $(G,\lambda)$ on the graph $G$ is a
continuous-time Markov process whose state space is a suitable
subset of $\N^V.$ The process is described through the number
$\eta(t,v)$ of particles at vertex $v$ at time $t$ and evolves
according the following rules: for each $v\in V$
\begin{eqnarray*}
  \eta(t,v) &\rightarrow & \eta(t,v)-1 \mbox{ at rate } \eta(t,v) \\
  \eta(t,v) &\rightarrow & \eta(t,v)+1 \mbox{ at rate } \lambda \sum_{u: u\sim
  v}\eta(t,u),
\end{eqnarray*}
where $\lambda$ is a fixed parameter and $u\sim v$ denotes that
$u$ is a neighbor of $v.$ In words, each particle dies at rate $1$
and gives birth to new particles at each neighboring vertex at
rate $\lambda.$ Let $o\in V$ be some distinguished vertex of $G$
and denote $\P_\orig$ the probability measure of the process
started with one particle at $\orig$ at time $0.$ One says the
branching random walk $(G,\lambda)$ survives locally if
$$\P_\orig\left(\eta(t,\orig)=0 \mbox{ for sufficiently large }
t\right)<1$$ or equivalently

$$\P_\orig\left( \exists (t_n)_{n\in\N}:~ t_n\to\infty \mbox{ and }
\eta(t_n,\orig)>0    \right)=0.$$

Since this is equivalent to the fact that the probability that
infinitely many particles jump to $\orig$ is zero, the question of
local survival can be answered with an appropriate BMC. To this
end let $(V,\bo)$ be any BMC with mean substitution
$m(u,v):=\lambda$ if $u\sim v$ and $m(u,v):=0$ otherwise. Hence,
$M=\lambda \cdot A,$ where $A$ is the adjacency matrix of the
graph $G=(V,E).$ Theorem \ref{thm:2} gives that the BMC $(V,\bo)$
is transient if and only if $\lambda\leq 1/\rho(A).$

It is now straightforward to obtain the following result that
strengthen Proposition 2.5 and Lemma 3.1 in \cite{pemantle:01}
where the behavior in the critical case, $\lambda=1/\rho(A),$ was
not treated.

\begin{cor}
The branching random walk $(G,\lambda)$ survives locally if and
only if $\lambda> 1/\rho(A).$
\end{cor}

\section{BRWRE on $\Z^d$}
We turn now to branching random walk in random environment (BRWRE)
on $\Z^d$ and see how the results of the preceding section apply
to this model. First, we define the model.
\subsection{The model}
Let $\U\subset\Z^d$ be a finite generator of the group $\Z^d.$
Define
$$\V:=\left\{ v=(v_y,~y\in\U): v_y\in\N,~\sum_{y\in\U} v_y\geq
1\right\}.$$  Furthermore, let us define the probability space
that describes the random environment. To this end, let
$$\M:=\left\{\omega=\left(\omega(v),~v\in \V\right):~ \omega(v)\geq 0
\mbox{ for all } v\in\V,~\sum_{v\in\V}\omega(v)=1 \right\}.$$ be
the set of all probability measures $\omega$ on $\V$ and let $Q$
be a probability measure on $\M.$ For each $x\in\Z^d$ a random
element $\omega_x\in\M$ is chosen according to $Q$ independently.
Let $\Theta$ be the corresponding product measure with
one-dimensional marginal $Q$ and denote $\K:=supp(Q)$ the support
of the marginal. The collection $\bm\omega = (\omega_x,~x\in\Z^d)$
is called a realization of  the random environment $\Theta.$ Each
realization $\bm\omega$ defines a BMC $(\Z^d,\bm\omega)$ and we
denote  $\P_{\bm\omega}$ the corresponding probability measure.
Throughout the paper we will assume the following condition on $Q$
that ensures the irreducibility of our process:
\begin{equation}\label{cond:irred}
 Q\left\{ \sum_y \omega_0(y)>\eps \quad \forall y\in \S
 \right\}=1 \quad \mbox{ for some }\eps>0,
\end{equation}
where $\S\subset\U$ is some  generating set of $\Z^d.$ For
example, we can choose $\S:=\{\pm e_i, 1\leq i\leq d\},$ where
$e_i$ is  the $i$th coordinate vector of $\Z^d.$ The uniform
condition in (\ref{cond:irred}) is used for Lemma \ref{lem:3} and
Lemma \ref{lem:4} where we need that the BMC $(\Z^d,\bs)$ is
irreducible for all realization $\bs=(\sigma_x)_{x\in \Z^d}$ with
$\sigma_x\in\mathcal{K}.$

\subsection{Transience or strong recurrence}
Due to condition (\ref{cond:irred}) $\Theta$-almost every
$\bm\omega$ defines an irreducible matrix $M_{\bo}.$ Hence, with
Lemma \ref{lem:2a}
$$\rho(M_{\bo})=\limsup_{n\to\infty}
\left({m_{\bo}}^{(n)}(x,x)\right)^{1/n},\quad \forall x\in \Z^d.$$

Furthermore, we have that the translation $\{T_z:~ z\in \Z^d\}$
acts ergodically as a measure preserving transformation on our
environment. Together with the fact that $\limsup_{n\to\infty}
\left({m_{\bo}}^{(n)}(x,x)\right)^{1/n}$ does not depend on $x$
this implies that $\log \rho(M_{\bo})$ is equal to a constant for
$\Theta$-a.a. realizations $\bo.$ Eventually, $\rho=\rho(M_{\bo})$
for $\Theta$-a.a. realizations $\bo$ and some $\rho,$ that we call
the spectral radius of the BRWRE. Together with Theorem
\ref{thm:2} this immediately implies that the BRWRE is either
transient for $\Theta$-a.a. realizations or recurrent for
$\Theta$-a.a. realizations. We have even the stronger result,
compare with \cite{comets:05}:
\begin{thm}\label{thm:0-1}
We have either
\begin{itemize}
    \item for $\Theta$-a.a. realizations $\bo$ the BRWRE is strongly
    recurrent:
    $$ \alpha(\bo,x):=\P_{\bo,x}\left( \sum_{n=1}^\infty
\sum_{i=1}^{\eta(n)}\mathbf{1}\{x_i(n)=x\}=\infty\right)=1\quad
\forall x\in X,\mbox{ or}$$
    \item for $\Theta$-a.a. realizations $\bo$ the BRWRE is transient:
    $$ \P_{\bo,x}\left( \sum_{n=1}^\infty
\sum_{i=1}^{\eta(n)}\mathbf{1}\{x_i(n)=x\}=\infty\right)=0\quad
\forall x\in X.$$
\end{itemize}
\end{thm}
\begin{proof}
It remains to show that $\rho>1$ implies $\alpha(\bo,x)=1$ for all
$x\in X.$ For $\Theta$-a.a. $\bo$ there exists some $Y\subset\Z^d$
such that $\rho(M_{\bo_Y})>1$ and the corresponding multi-type
Galton-Watson process is supercritical and survives with positive
probability. We start the process in $x\in X.$ Since the random
environment is iid, it is easy to construct a sequence of
independent multi-type Galton-Watson processes whose extinction
probability are bounded away from $1,$ we refer to
\cite{comets:05} and \cite{mueller06} for more details. At least
one of these processes survives and infinitely many particles
visit the starting position $x,$ i.e. $\alpha(\bo,x)=1.$
\end{proof}

\subsection{The spectral radius of BRWRE and the transience criterion}

We first give the transience criterion.

\begin{thm}\label{thm:3}
The BRWRE is transient for $\Theta$-a.a. realizations if
\begin{equation}\label{eq:cond}
\sup_{m\in\hat{\mathcal{K}}} \inf_{\theta\in\R^d} \left( \sum_s
e^{\langle\theta, s\rangle} m(s)\right)\leq 1.
\end{equation}
Otherwise it is strongly recurrent for $\Theta$-a.a. realizations.
\end{thm}

\begin{rem}
The fact that condition (\ref{eq:cond}) is equivalent to
\emph{Condition L} of \cite{comets:05} follows through
straightforward calculation and the fact that the $\sup$ and
$\inf$ in (\ref{eq:cond}) are attained.
\end{rem}

The remaining part of the paper is devoted to the identification
of $\rho$ in order to show  the explicit criterion for transience,
Theorem \ref{thm:3}. The first observation is that the spectral
radius is \emph{as large as possible}, compare with Lemma
\ref{lem:2}, then that it only depend on the convex hull of the
support $\K,$ compare with Lemma \ref{lem:3}, and the last one
that it equals the spectral radius of an appropriate homogeneous
BMC, compare with Lemma \ref{lem:4}.

The first Lemma is a straightforward generalization of
\cite{mueller06} or alternatively follows from the proof of Lemma
\ref{lem:3}.
\begin{lem}\label{lem:2}We have $$
\rho:=\rho(M_{\bo})=\sup_{\bs} \rho(M_{\bs}) ~for~
\Theta\mbox{-a.a.}~\bo,$$ where the $\sup$ is over all possible
collections $\bs=(\sigma_x)_{x\in \Z^d}$ with
$\sigma_x\in\mathcal{K}.$
\end{lem}

Furthermore, the spectral radius does only depend on the convex
hull $\hat\K$ of the support $\K$ of $Q,$ compare with
\cite{varadhan2003b} where this is done for random walk in random
environment.

\begin{lem}\label{lem:3}
We have $$ \rho:=\rho(M_{\bo})=\sup_{\hat{\bs}} \rho(M_{\hat\bs})
~for~ \Theta\mbox{-a.a.}~{\bo},$$ where the $\sup$ is over all
possible collections $\hat{\bs}=(\hat\sigma_x)_{x\in \Z^d}$ with
$\hat\sigma_x\in\hat{\mathcal{K}}.$
\end{lem}

\begin{proof}
In order to see this recall that for any irreducible kernel $K$ we
have, due to Lemma \ref{lem:1} and Equation (\ref{eq:specapprox}),
$$\rho(K)=\min\{t>0:~\exists\, f(\cdot)>0\mbox{ such that }Kf\leq
t f\}$$ and
\begin{equation} \rho(K)=\sup_Y \rho(K_Y),
\end{equation}
where the supremum is over finite and irreducible subsets
$Y\subset \Z^d.$ Let $F\subset \Z^d$ be an irreducible subset with
respect to $M$ and define as usual the complement
$F^c:=\Z^d\setminus F,$ the boundary $\partial F:=\{x\in
F:~m(x,y)>0\mbox{ for some } y\in F^c\}$ and the inner points
$F^\circ=F\setminus\partial F$  of the set $F.$ The key point is
now to consider the equation
\begin{eqnarray}\label{eq:bell1}
\sum_y m(x,y) f(y)&=&t\cdot f(x)\quad \forall x\in F^\circ
\end{eqnarray} with boundary condition
 $f(y) =1 \quad \forall y\in \partial F.$
In order to give a solution of  equation (\ref{eq:bell1}), we
consider a modified process where we identify the border of $F$
with a single point, say $\triangle.$ Let  $\widetilde M$ be the
finite kernel over $\widetilde F:=F^\circ\cup\{\triangle\}$ with
$\widetilde m(x,y):=m(x,y)$ for all $x,y\in F^\circ$, $\widetilde
m(x,\triangle):=\sum_{y\in \partial F} m(x,y)$ for all $x\in
F^\circ$ and $\widetilde m(\triangle, x)=0$ for all $x\in
\widetilde F.$ Observe that $\widetilde M$ is finite with
absorbing state $\triangle.$ Then the function
$$\widetilde f_F(t,x):= \sum_{k=0}^\infty \widetilde m^{(k)}(x,\triangle)
(1/t)^k$$ is the unique solution of
\begin{eqnarray}\label{eq:bell2}
\sum_y \widetilde m(x,y) f(y)&=&t\cdot f(x)\quad \forall x\in
F^\circ
\end{eqnarray} with boundary condition
 $f(\triangle) =1.$

One can think of $\widetilde f_F(t,x)$ as the expected number of
particles visiting $\partial F$ for the first time in their
ancestry line in the multi-type Galton-Watson process with first
mean $1/t\cdot \widetilde M$ and one original particle at $x\in
F.$ Furthermore, we have that
$$\rho(\widetilde M)=\inf\{t>0:~ \widetilde f_F(t,x)<\infty\quad \forall x\in F\},$$
since $R(\widetilde M)=1/\rho(\widetilde M)$ is the convergence
radius of $\widetilde G(x,x|z):=\sum_{k=0}^\infty \widetilde
m^{(k)}(x,x)z^k$ for all $x\in F^\circ.$ Since $m_F(x,y)=m(x,y)$
for all $x,y\in F$ we have that the convergence radius of
$\widetilde G(x,x|z)$ equals the one of
$G_F(x,x|z):=\sum_{k=0}^\infty m_F^{(k)}(x,x)z^k$ for all $x\in
F^\circ.$ Eventually,
$$\rho( M_F)=\inf\{t>0:~ f_F(t,x)<\infty\quad \forall x\in F\},$$
where $f_F(t,x):=\tilde f_F(t,x)$ for all $x\in F^\circ$ and
$f_F(t,y):=\tilde f_F(t,\triangle)$ for all $y\in \partial F.$

The last step is now to determine for every finite set $F$ the
largest spectral radius over all possible choices of $M_F$ with
$m(x,\cdot)\in \mathcal{K}$  and show that this value does not
change if we maximize over all possible choices $M_F$ with
$m(x,\cdot)\in \hat{\mathcal{K}}.$ We consider the following
dynamical programming problem:

\begin{eqnarray}\label{eq:bell3}
\sup_{m(x,\cdot)\in \mathcal{K}}\sum_y m(x,y) f(y)&=&t\cdot
f(x)\quad \forall x\in F^\circ
\end{eqnarray} with boundary condition
 $f(y) =1 \quad \forall y\in \partial F.$ The goal of the optimization
 problem is to maximize $f_F(t,x)$ over the possible choices of
 $M_F.$ Observe that there will be a maximal value of $t^*$ such
 that the solution $f^*$ of the optimization problem is finite for all $t>t^*.$ The value $t^*$ is now
 equal to the largest spectral radius $\rho(M_F)$ we can achieve. We conclude
 with the observation, that $t^*$ does not change if we replace
 $\mathcal{K}$ by its convex hull $\hat\K$ and the fact that
 $\rho(M)=\sup_F \rho(M_F).$
\end{proof}

The next step is to show that the spectral radius $\rho$ of the
BRWRE equals the spectral radius of some homogeneous BMC and can
therefore be calculated explicitly. We generalize the
argumentation of   \cite{mueller06} and \cite{varadhan2003b}.

\begin{lem}\label{lem:4}
For a RWRE on $\Z^d$ we have for $\Theta$-a.a. realizations
${\bo}$
\begin{eqnarray}
   \rho(M_{\bo}) &=& \sup_{m\in\hat{\mathcal{K}}} \rho(M^h_m) \\
   &=& \sup_{m\in\hat{\mathcal{K}}} \inf_{\theta\in\R^d}
   \left( \sum_{s\in \U} e^{\langle
   \theta, s\rangle} m(s)\right),
\end{eqnarray}
where $M^h_m$ is the transition matrix of the BMC with
$m(x,x+s)=m(0,s)=:m(s)$ for all $x\in\Z^d,~s\in \U.$
\end{lem}
\begin{proof}
The second equality is more or less standard. It  follows for
example  from the fact that $\rho(P)=\exp(-I(0)),$ where
$I(\cdot)$ is the rate function of the large deviations of the
random walk on $\Z^d$ with transition probabilities $P:=M/\sum_s
m(s).$ Since, due to Lemma \ref{lem:3}, we have $\rho(M_{\bo})\geq
\sup_{m\in\hat{\mathcal{K}}} \rho(M^h_m),$  it remains to show
$$\rho(M_{\bo})\leq  \sup_{m\in\hat{\mathcal{K}}}
\inf_{\theta\in\R^d} \left(\sum_s e^{\langle  \theta, s\rangle}
m(s)\right)$$ for $\Theta$-a.a. realizations ${\bo}.$ Observing
that the function $\phi(m(\cdot),\theta):=\left(\sum_s e^{\langle
\theta, s\rangle} m(s)\right)$ is convex in $\theta$ and linear in
$m(\cdot)$, we get by a standard minimax argument, compare with
\cite{sion:58}, that
$$\sup_{m\in\hat{\mathcal{K}}} \inf_{\theta\in\R^d} \sum_s e^{\langle
\theta, s\rangle} m(s)= \inf_{\theta\in\R^d}
\sup_{m\in\hat{\mathcal{K}}} \sum_s e^{\langle \theta, s\rangle}
m(s)=:c.$$ Let $\eps>0$ and $\theta\in \R^d$  such that
   $$ \sup_{m\in\hat{\mathcal{K}}} \sum_s e^{\langle
\theta, s\rangle} m(s) \leq c (1+\eps). $$ By induction we have
for any realization ${\bo}:$
$$ \sum_{x_n\in\Z^d}e^{\langle \theta, x_n\rangle} m_{\bo}(0,x_n) \leq
   (c (1+\eps))^n.$$
 Therefore by observing only $x_n=0:$
    $$ m_{\bo}^{(n)}(0,0)\leq (c (1+\eps)) ^n,$$ and hence
    $\rho(M_{\bo})\leq c (1+\eps)$ for all $\eps>0.$
\end{proof}

\begin{small}
%\addcontentsline{toc}{chapter}{Bibliography}
\bibliography{bib}

\def\cprime{$'$}
\begin{thebibliography}{10}

\bibitem{benjamini:94}
I.~Benjamini and Y.~Peres.
\newblock Markov chains indexed by trees.
\newblock {\em Ann. Probab.}, 22(1):219--243, 1994.

\bibitem{comets98}
F.~Comets, M.~V. Menshikov, and S.~Yu. Popov.
\newblock One-dimensional branching random walk in a random environment: a
  classification.
\newblock {\em Markov Process. Related Fields}, 4(4):465--477, 1998.

\bibitem{comets:05}
F.~Comets and S.~Popov.
\newblock On multidimensional branching random walks in random environment.
\newblock {\em Ann. Prob.}, 35(1):68--114, 2007.

\bibitem{gantert:04}
N.~Gantert and S.~Müller.
\newblock The critical branching {M}arkov chain is transient.
\newblock {\em Markov Proc. and rel. Fields.}, 12:805--814, 2007.

\bibitem{harris:63}
T.~E. Harris.
\newblock {\em The theory of branching processes}.
\newblock Die Grundlehren der Mathematischen Wissenschaften, Bd. 119.
  Springer-Verlag, Berlin, 1963.

\bibitem{machado:00}
F.~P. Machado and S.~Yu. Popov.
\newblock One-dimensional branching random walks in a {M}arkovian random
  environment.
\newblock {\em J. Appl. Probab.}, 37(4):1157--1163, 2000.

\bibitem{machado:03}
F.~P. Machado and S.~Yu. Popov.
\newblock Branching random walk in random environment on trees.
\newblock {\em Stochastic Process. Appl.}, 106(1):95--106, 2003.

\bibitem{menshikov:97}
M.~V. Menshikov and S.~E. Volkov.
\newblock Branching {M}arkov chains: Qualitative characteristics.
\newblock {\em Markov Proc. and rel. Fields.}, 3:225--241, 1997.

\bibitem{mueller06}
S.~Müller.
\newblock Recurrence and transience for branching random walks in an iid random
  environment.
\newblock {\em Markov Proc. and rel. Fields.}, to appear in 2007.

\bibitem{pemantle:01}
R.~Pemantle and A.~Stacey.
\newblock The branching random walk and contact process on {G}alton-{W}atson
  and nonhomogeneous trees.
\newblock {\em Ann. Probab.}, 29(4):1563--1590, 2001.

\bibitem{pruitt}
W.~E. Pruitt.
\newblock Eigenvalues of non-negative matrices.
\newblock {\em Ann. Math. Statist.}, 35:1797--1800, 1964.

\bibitem{schinazi:93}
R.~Schinazi.
\newblock On multiple phase transitions for branching {M}arkov chains.
\newblock {\em J. Statist. Phys.}, 71(3-4):507--511, 1993.

\bibitem{sion:58}
M.~Sion.
\newblock On general minimax theorems.
\newblock {\em Pacific J. Math.}, 8(1):171--176, 1958.

\bibitem{stacey:03}
A.~Stacey.
\newblock Branching random walks on quasi-transitive graphs.
\newblock {\em Combin. Probab. Comput.}, 12(3):345--358, 2003.
\newblock Combinatorics, probability and computing (Oberwolfach, 2001).

\bibitem{sznitman:02}
A.-S. Sznitman.
\newblock An effective criterion for ballistic behavior of random walks in
  random environment.
\newblock {\em Probab. Theory Related Fields}, 122(4):509--544, 2002.

\bibitem{sznitman:03}
A.-S. Sznitman.
\newblock On new examples of ballistic random walks in random environment.
\newblock {\em Ann. Probab.}, 31(1):285--322, 2003.

\bibitem{varadhan2003b}
S.~R.~S. Varadhan.
\newblock Large deviations for random walks in a random environment.
\newblock {\em Comm. Pure Appl. Math.}, 56(8):1222--1245, 2003.

\bibitem{woess}
W.~Woess.
\newblock {\em Random walks on infinite graphs and groups}, volume 138 of {\em
  Cambridge Tracts in Mathematics}.
\newblock Cambridge University Press, Cambridge, 2000.

\end{thebibliography}
\end{small}

\bigskip
 \noindent
\begin{tabular}{l}
Sebastian M\"uller \\
 Institut f\"ur Mathematische Statistik\\
 Universit\"at Münster\\
 Einsteinstr. 62  \\
 D-48149 Münster\\
 Germany\\
{\tt Sebastian.Mueller@math.uni-muenster.de}\\
\end{tabular}

\end{document}